\documentclass[12pt, a4paper]{amsart}
\usepackage{amsfonts}

\usepackage{amscd,amssymb}
\usepackage[pdftex]{graphicx}
\usepackage{times}
\usepackage[T1]{fontenc}
\usepackage{ajmaa}
\usepackage[
    pdftex,
    pdfpagemode=None,
    bookmarksopen=true,
    baseurl={http://ajmaa.org/}
    pdfstartview={FitH}
]{hyperref}

\newcommand{\J}[1]{\mathcal{J}\left(#1,X\sim\mathcal{P}\right)}
\newcommand{\E}[1]{\mathbb{E}\left[#1\right]}
\def \jg {\J{f}}
\def \distr {\mathcal{P}}
\def \jgf {\E{f\left(X\right)}-f\left(\E{X}\right)}

\theoremstyle{plain}

\newtheorem{cor}{Corollary}
\newtheorem{prop}{Proposition}

\numberwithin{equation}{section}

\issueinfo{16}{2}{14}{2019}
\dateinfo{6 November, 2018}{11 October, 2019}{11 November, 2019}
\copyrightinfo{2019}{Austral Internet Publishing}
\editor{Not Used}

\usepackage{footnote}

\input{tcilatex}

\begin{document}
\title[The Jensen Gap]{Bounds on the Jensen Gap, and Implications for Mean-Concentrated Distributions}
\author{Xiang Gao, Meera Sitharam, Adrian E. Roitberg}
\address{Department of Chemistry\\
and Department of Computer \& Information Science \& Engineering\\
University of Florida\\
Gainesville, FL 32611, USA.}
\email{\href{mailto: Xiang Gao <qasdfgtyuiop@gmail.com>}{qasdfgtyuiop@gmail.com}}
\urladdr{\url{https://scholar.google.com/citations?user=t2nOdxQAAAAJ}}
\date{13 October, 2019.}

\begin{abstract}
This paper gives upper and lower bounds on the gap in Jensen's inequality, i.e., the difference between the expected value of a function of a random variable and the  value of the function at the expected value of the random variable.
The bounds depend only on growth properties of the function and specific moments of the random variable. The bounds are particularly useful for distributions that are concentrated around the mean, a commonly occurring scenario such as the average of i.i.d. samples and in statistical mechanics.
\end{abstract}

\maketitle

\section{Introduction}

The widely used Jensen's inequality for convex functions, attributed to Danish
mathematician Johan Jensen,  dates back to 1906\cite{jensen1906fonctions}. 
The literature contains numerous bounds on the Jensen gap, defined as $\jg = \jgf$, where $X$ is a random variable with distribution $\distr$, and the function $f$ might be convex or nonconvex \cite{zlobec2004jensen}.

Example consequences and applications of the known bounds are: a number of famous classical
inequalities such as the generalized mean inequality,
and a special case the inequality of arithmetic and geometric means,
the H\"older's inequality, etc.  \cite{steele2004cauchy};  
commonly used results in information theory, e.g. non-negativity of Kullback-Leibler divergence
 \cite{cover2012elements}; 
 variational bounds for negative log likelihood used 
in statistics and machine learning methods such as the expectation maximization
algorithm\cite{dempster1977maximum},  and variational inference\cite{wainwright2008graphical}. 

Computing a hard-to-compute $\E{f(X)}$ appears in theoretical estimates in a variety of scenarios from statistical mechanics to machine learning theory. A common approach to tackle this problem is to make the approximation $\E{f(X)}\approx f\left(\E{X}\right)$ (for example $\left<\frac{1}{X}\right>\approx\frac{1}{\left<X\right>}$), and then show that the error, i.e.,  the Jensen gap, would be small enough for the application. Since the error itself is as hard to compute as $\E{f(X)}$, inequalities on the Jensen gap would help by giving easy-to-compute bounds.  Moments are commonly used to characterize distributions of random variables because of their relative ease to compute for many distributions. By establishing the connection between the Jensen gap and moments, we create a powerful tool for error estimation based on moment estimates.

As a concrete scenario, the Jensen gap has many useful interpretations in statistical mechanics such as the difference of average non-equilibrium
work and change of free energy, an important quantity to characterize the deviation of a thermodynamical process from a quasi-static process, as in Jarzynski equality\cite{PhysRevLett.78.2690}, and the fluctuation of thermodynamical quantities around their ensemble average, which is of common interest in physics. In machine learning theory, stochastic gradient descent is employed to 
minimize the so-called loss function, when learning the parameters of a function from a parametrized family; in this case the training inputs are sampled from a distribution and the loss function is an expectation, to be minimized with respect to the parameters.

In such scenarios, a common type of  random variable  has a distribution concentrated around its mean as described below.
\begin{enumerate}
\item  In estimating $\xi=f\left(\E{X}\right)$, an empirical average from samples is often used as an estimation of expectation, i.e. $\E{X}\approx\bar{X}=\frac{1}{\left|\mathcal{M}\right|}\sum_{X_i\in\mathcal{M}} X_i$ and $\xi\approx\hat\xi=f\left(\bar{X}\right)$. The bias of $\hat\xi$, i.e. $\mathbb{E}_{\mathcal{M}}[\hat\xi]-\xi$ is given by the Jensen gap where the random variable $\bar{X}$ has a distribution concentrated around its mean. The asymptotic growth behavior of the Jensen gap therefore gives an idea how fast we can push the bias to zero by increasing $N$.
\item Random variables with a distribution concentrated around the mean are very common in statistical mechanics. Since the number of particles in the system is usually the order of Avogadro constant $N_A\sim 10^{23}$, the distribution is so sharp that all the Jensen gaps become negligible. However, this is not the case in computer simulation or microscopic experiments, which usually have a much smaller system size. The asymptotic growth behavior of thermodynamic fluctuation (defined as the Jensen gap of function $\sqrt\cdot$ of random variable $E^2$) with the system size guides the simulation/experiment setup.
\end{enumerate}
Moments play an important role in studying random variables with a distribution concentrated on the mean, especially when the random variable is an empirical average of i.i.d random variables\cite{packwood2011moments}.  Our results will use moments to express the asymptotic growth behavior of the Jensen gap.

Next we give an elementary example to illustrate the inspiration behind our results, which establish the connection between the Jensen gap and the (absolute centered) moment $\sigma_{p}=\sqrt[p]{\E{\left|X-\mu\right|^{p}}}$, where $\mu=\E{X}$ is the expectation of random variable $X$.
Assume that for $\alpha>0$, $f\left(x\right)$ is $\alpha$-H\"older continuous over $\mathbb{R}$,
i.e. there exists a positive number $M$ such that for any $x\in\mathbb{R}$,
$\left|f\left(x\right)-f\left(\mu\right)\right|\leq M\left|x-\mu\right|^{\alpha}$. Then we have

\begin{multline}
\left|\jgf\right|\leq\int\left|f\left(X\right)-f\left(\mu\right)\right|d\mathcal{P}\left(X\right)\\
\leq M\int\left|x-\mu\right|^{\alpha}d\mathcal{P}\left(X\right)\leq M\sigma_{\alpha}^{\alpha}\label{eq:trivial-upper-bound}
\end{multline}
Similarly, if $f\left(x\right)-f\left(\mu\right)\geq M\left|x-\mu\right|^{\alpha}$ or $f\left(\mu\right)-f\left(x\right)\geq M\left|x-\mu\right|^{\alpha}$,
we can obtain an elementary lower bound on the Jensen gap
\begin{multline}
 \left|\jgf\right|=\int\left|f\left(X\right)-f\left(\mu\right)\right|d\mathcal{P}\left(X\right)\\
\geq M\int\left|x-\mu\right|^{\alpha}d\mathcal{P}\left(X\right)\geq M\sigma_{\alpha}^{\alpha}\label{eq:trivial-lower-bound}
\end{multline}
Our main results generalize these two elementary bounds as described in the next section.

\subsection{Contribution and Comparison}
We prove an upper and lower bound on the Jensen gap, summarized below, and demonstrate their tightness.  In the following, "upper bound of  $A$" means $\left|\jg\right|\leq A$ which means $-A\leq\jg\leq A$, while "lower bound of  $A$" means either $\jg\geq A$ or $-\jg\geq A$.
\begin{itemize}
\item For functions that approach $f\left(\mu\right)$ at $x\to\mu$ no
slower than $\left|x-\mu\right|^{\alpha}$, and grow as $x\to\pm\infty$
no faster than $\pm\left|x\right|^{n}$ for $n\geq\alpha$, 
\[
 \left|\jgf\right| \leq M\left(\sigma_{\alpha}^{\alpha}+\sigma_{n}^{n}\right)\leq M\left(1+\sigma_{n}^{n-\alpha}\right)\sigma_{n}^{\alpha}
\]
where $M=\sup_{x\neq\mu}\frac{\left|f\left(x\right)-f\left(\mu\right)\right|}{\left|x-\mu\right|^{\alpha}+\left|x-\mu\right|^{n}}$.
 This implies that  $\jgf$
 approaches $0$ no slower than $\sigma_{n}^{\alpha}$ as $\sigma_{n}\to0$.
\item For functions that either decrease or increase (but do not decrease on one side and increase on the other) to $f\left(\mu\right)$ as $x\to\mu$
no faster than $\left|x-\mu\right|^{\alpha}$, and grow to infinity at
$x\to\infty$ no slower than $\left|x\right|^{\beta}$ for $0\leq\beta\leq\alpha$,
\[
 \left|\jgf\right|\geq M\frac{\sigma_{\alpha/2}^{\alpha}}{1+\sigma_{\alpha-\beta}^{\alpha-\beta}}
\]
where $M=\inf_{x\neq\mu}\left\{ \left[f\left(x\right)-f\left(\mu\right)\right]\cdot\left(\frac{1}{\left|X-\mu\right|^{\beta}}+\frac{1}{\left|X-\mu\right|^{\alpha}}\right)\right\} $.
 This implies that  $ \jgf$
decreases to $0$ no faster than $\sigma_{\alpha/2}^{\alpha}$ as $\sigma_{\alpha/2}\to0$
as long as $\sigma_{\alpha-\beta}$ does not grow to infinity at the
same time.
\end{itemize}
Although neither our upper bounds nor our lower bounds require the function to be convex or concave, the condition in our lower bound is naturally satisfied by convex or concave functions as we will show in Section \ref{subsec:lower-convex}. 

In order to illustrate the flavor of the main results, we give simple examples that are direct consequences. We also compare the consequences of our main results with known lower bounds\cite{abramovich2004refining,simic2009jensen,walker2014lower,abramovich2016some,zlobec2004jensen,liao2017sharpening} and upper bounds\cite{costarelli2015sharp,Simic2009}. A major advantage of our result over these known results is their relative generality: our conditions on the function, its domain, and the distribution are weak (for example we do not require the function to be convex, and we do not require the distribution to be discrete). Among the above-mentioned bounds, \cite{simic2009jensen} and \cite{Simic2009} are only for discrete distributions, and  are omitted from the comparisons below. Besides the bounds as listed above, the Jensen's gap can also be estimated by Jensen-Ostrowski type inequalities\cite{cerone2015jensen, dragomir2016jensen, cerone2015inequalities, dragomir2016ostrowski, cerone2017jensen}.

\begin{example} 
Consider the Jensen gap of $\sin\left(x\right)$ and random variables with mean at $0$. Observe that $\sin\left(x\right)$ has a power series $\sin\left(x\right)=x-\frac{x^{3}}{6}+\frac{x^{5}}{120}+\cdots$, and 
by choosing $\alpha=n=1$, we get $\left|\J{\sin}\right|\leq \sigma_1^1$.
Also, since $\sin'\left(0\right)=1\neq0$, we can obtain a different result by studying $g\left(x\right)=\sin\left(x\right)-x$ (which has the same gap behavior, discussed  in Section \ref{subsec:upper-shifting} and Section \ref{subsec:lower-convex}) instead.
This time, by choosing $\alpha=n=3$, we can see that $\left|\J{\sin}\right|\leq \frac{\sigma_3^3}{6}$. If we are interested in the asymptotic behavior of the Jensen gap when the distribution is concentrated around the mean, we can conclude immediately that $\left|\J{\sin}\right|$ decreases to $0$ no slower than $\sim\sigma_3^3$ and $\sim\sigma_1^1$.
It is also possible to choose $\alpha=n=2$ and obtain $\left|\J{\sin}\right|\leq\frac{\sigma_{2}^{2}}{\pi}$.
Although this result is not as good as the $\sigma_{3}^{3}$ version
in terms of asymptotic behavior for non-heavy-tailed distributions
such as Gaussian distribution and Laplace distribution, the second
moment is usually more available than the third moment.
Our lower bound result is not useful in this example. In fact, since $\sin(x)$ is odd, any even distribution $\distr$ will result in a zero Jensen gap regardless of its moments. That is, it is impossible to obtain a non-trivial bound that is a function of only moments.

Results in \cite{costarelli2015sharp,walker2014lower} require the function to be convex, and are therefore not useful for this example. Since the domain is not $\left[0,A\right)$ or $\left(0,A\right]$ as required by \cite{abramovich2016some}, or $\left[0,+\infty\right)$ as required by \cite{abramovich2004refining}, these results do not apply either.
The result in \cite{liao2017sharpening} gives the same bound as our
$\sigma_{2}^{2}$ version, which is better than the $\J{\sin}\geq-\frac{\sigma_{2}^{2}}{2}$
given by \cite{zlobec2004jensen}. The fact that \cite{liao2017sharpening}
gives the same bound as ours is not just a coincidence, but can be 
attributed to the connections between our results and \cite{liao2017sharpening}
as described in Section \ref{subsec:upper-shifting}.
\end{example}
\begin{example}
Consider $\cos\left(x\right)$ and random variables with mean at $0$.
Observe that $\cos\left(x\right)$ has a power series $\cos\left(x\right)=1-\frac{x^{2}}{2}+\cdots$, we can choose $\alpha=n=2$ and see that $\left|\J{\cos}\right|\leq\frac{\sigma_2^2}{2}$. If we are interested in the asymptotic behavior, we can conclude that $\left|\J{\cos}\right|$ will decrease no slower than $\sim\sigma_2^2$, i.e. the variance of the distribution.
Again, our lower bound result is not useful in this example. In fact, although non-trivial, it is possible to construct  a probability distribution $\distr$ that makes the Jensen gap equal  $0$ and has arbitrary moments,$^1$ \footnotemark[1] \footnotetext[1]{$^1$To construct such a probability distribution, we can choose a discrete $\distr$ whose support is a subset of $\left\{\pm 2\pi k\middle|k\in\mathbb{N}\right\}$. By appropriate choice of probability values at discrete points, it is possible to make $\distr$ have any desired set of moments.} that is, it is impossible to obtain a non-trivial bound that is a function of only moments.

Results in \cite{costarelli2015sharp,walker2014lower} require the function to be convex, therefore not useful for this example. Since the domain is not $\left[0,A\right)$ or $\left(0,A\right]$ as required by \cite{abramovich2016some}, or $\left[0,+\infty\right)$ as required by \cite{abramovich2004refining}, these results do not apply.
Both our result and \cite{liao2017sharpening,zlobec2004jensen} are
able to get $\J{\cos}\geq-\frac{\sigma_{2}^{2}}{2}$. Since $\cos\left(x\right)-1\leq0$,
it is not hard to see that $\J{\cos}\leq0$, which is also given
by the result of \cite{liao2017sharpening}.
\end{example}
\begin{example}
Consider $\log\left(x\right)$ and random variable $X\in\left[a,+\infty\right)$ with $a>0$ that has $\E{X}=1$.
Since $\log'\left(x\right)=1\neq0$, we study $g\left(x\right)=\log\left(x\right)-\left(x-1\right)$ instead, which preserves gap behavior as in the $\sin$ example above.
Note that $\log\left(x\right)=\left(x-1\right)-\frac{1}{2}\left(x-1\right)^{2}+\cdots$. By choosing $\alpha=n=2$, we have
\[
\left|\J{\log}\right|\leq\frac{a-1-\log(a)}{(1-a)^2}\sigma_2^2
\]
i.e. $\left|\J{\log}\right|$
will decrease no slower than $\sim\sigma_2^2$, i.e. the variance of the distribution, as $\sigma_2\to0$. Since the $\log$ function is concave, our lower bound is useful (see Section \ref{subsec:lower-convex}). Choosing $\alpha=2$ and $\beta=1$, we get
\[
-\J{\log}\geq \frac{1}{2}\cdot\frac{\sigma_1^2}{1+\sigma_1}
\]
whereby the Jensen gap approaches 0 no faster than $\sigma_1^2$ as $\sigma_1\to0$.

The estimate given by \cite{costarelli2015sharp} is
\[
\J{\log}\leq \frac{1}{2a^2}\min_{c\geq a}\left\{\E{(X-c)^2}+(1-c)^2\right\}=\frac{\sigma_2^2}{2a^2}
\]
for $a<1$
\[
\frac{a-1-\log(a)}{(1-a)^2} < \frac{1}{2a^2}
\]
which means that we get a better result than \cite{costarelli2015sharp}.   Since the domain is not $\left[0,A\right)$ or $\left(0,A\right]$ as required by \cite{abramovich2016some}, or $\left(0,+\infty\right)$ as required by \cite{walker2014lower}, or $\left[0,+\infty\right)$ as required by \cite{abramovich2004refining}, these results do not apply. The result in \cite{zlobec2004jensen} in this case falls back to Jensen's inequality $\J{\log}\geq 0$. 
The result in \cite{liao2017sharpening} gives the same upper bound as ours and fall back to Jensen's inequality for the lower bound.
\end{example}
\begin{example}
Consider $f\left(x\right)=\sqrt{x}$ and random variables on $\left[0,+\infty\right)$ that has mean at $1$. Since $f'\left(1\right)=\frac{1}{2}\neq0$,
we study $g\left(x\right)=\sqrt{x}-\frac{x-1}{2}$ instead. Note that $\sqrt{x}=1+\frac{x-1}{2}-\frac{1}{8}(x-1)^2+\cdots$. By choosing $n=\alpha=2$, we see that $\left|\J{\sqrt\cdot}\right|\leq\frac{\sigma_2^2}{2}$, i.e. $\left|\J{\sqrt\cdot}\right|$
will decrease no slower than $\sim\sigma_2^2$, i.e. the variance of the distribution. Also, since $\sqrt{x}$ is concave, our lower bound is useful. By setting $\alpha=2$ and $\beta=1$, we get
\[
-\J{\sqrt\cdot}\geq \frac{1}{8}\cdot\frac{\sigma_1^2}{1+\sigma_1}
\]
whereby the Jensen gap will approach 0 no faster than $\sigma_1^2$.

Since the second order derivative is not bounded, \cite{costarelli2015sharp} does not apply to this example. Since $\frac{\sqrt{x}-\sqrt{0}}{x}$ is not defined on $0$ and does not have a power series on $0$, results in \cite{abramovich2016some} do not apply. Since the domain is not $\left(0,+\infty\right)$ as required by \cite{walker2014lower}, that result does not apply. Since $-\sqrt\cdot$ is superquadratic, \cite{abramovich2004refining} applies and has a result $-\J{\sqrt\cdot}\geq-\E{\sqrt{\left|X-1\right|}}=-\sigma_{1/2}^{1/2}$, which is not even an improvement of Jensen's inequality $-\J{\sqrt\cdot}\geq0$.  Again, the result in \cite{zlobec2004jensen} falls back to Jensen's inequality $-\J{\sqrt\cdot}\geq0$.
The result in \cite{liao2017sharpening} gives the same upper bound
as ours and fall back to Jensen's inequality for the lower bound.

\end{example}
\begin{example}
 Consider $f\left(x\right)=x^{4}$ and random variables that have mean at $1$. Since $f'\left(1\right)=4\neq0$,
we study $g\left(x\right)=x^{4}-4(x-1)$ instead. By choosing $\alpha=2$,
$n=4$, we see that $\left|\jg\right|\leq\frac{7+\sqrt{41}}{2}\left(1+\sigma_4^2\right)\sigma_4^2$, i.e. $\left|\jg\right|$
will decrease no slower than $\sim\sigma_4^{2}$. Also, since $f(x) = x^4$ is convex, our lower bound is useful. By choosing $\alpha=\beta=2$, we get $\jg\geq 2\sigma_1^2$ whereby the Jensen gap will decrease to 0 no faster than $\sigma_1^2$.

Since the second order derivative is not bounded, results in \cite{costarelli2015sharp} do not apply to this example. Since the domain is not $\left[0,A\right)$ or $\left(0,A\right]$ as required by \cite{abramovich2016some}, or $\left(0,+\infty\right)$ as required by \cite{walker2014lower}, or $\left[0,+\infty\right)$ as required by \cite{abramovich2004refining}, these results do not apply. Again \cite{zlobec2004jensen} falls back to Jensen's inequality $\J{f}\geq0$.
The Jensen gap of $x^{4}$ on $\mathbb{R}$ with $\mu=1$. The result in \cite{liao2017sharpening} gives a trivial upper bound $\J{f}\leq+\infty$
and a lower bound $\J{f}\geq2\sigma_{2}^{2}$. This lower bound
gives better numerical values and is usually easier to compute compared
with ours.
\end{example}

\section{First Main result: Upper bound}
\label{subsec:upper-main}
We first prove an upper bound on the Jensen gap  and discuss the tightness of this bound in Section \ref{subsec:upper-tightness}. Note that our  upper bound is useful even when the function $f$ is not convex.
Next, we show how to use shifts to expand the scope of our upper bound in Section \ref{subsec:upper-shifting}.

The upper bound in the following theorem holds for any probability distribution as long as the relevant moments are well defined.
\begin{theorem}
\label{thm:main-result-upper}If   $f:I\to\mathbb{R}$, where $I$ is a closed subset of $\mathbb{R}$ and $\mu\in I$,
satisfies the following conditions:
\begin{enumerate}
\item $f$ is bounded on any compact subset of $I$.
\item \textup{$\left|f\left(x\right)-f\left(\mu\right)\right|=O\left(\left|x-\mu\right|^{\alpha}\right)$
at $x\to\mu$ for $\alpha>0$ .}
\item $\left|f\left(x\right)\right|=O\left(\left|x\right|^{n}\right)$  as
$x\to\infty$ for $n\geq\alpha$ 
\end{enumerate}
then for a random variable $X$ with probability distribution $\distr$   and expectation $\mu$, the following inequality holds:

\begin{equation}
\left|\E{f\left(X\right)}-f\left(\mu\right)\right|\leq M\left(\sigma_{\alpha}^{\alpha}+\sigma_{n}^{n}\right)\leq M\left(1+\sigma_{n}^{n-\alpha}\right)\sigma_{n}^{\alpha}\label{eq:sigma-alpha-n-bound}
\end{equation}
where $M=\sup_{x\in I\backslash\{\mu\}}\frac{\left|f\left(x\right)-f\left(\mu\right)\right|}{\left|x-\mu\right|^{\alpha}+\left|x-\mu\right|^{n}}$
does not depend on the probability distribution $\mathcal{P}$.
\end{theorem}

\begin{proof}
 We begin by showing that $g\left(x\right)=\frac{\left|f\left(x\right)-f\left(\mu\right)\right|}{\left|x-\mu\right|^{\alpha}+\left|x-\mu\right|^{n}}$
is bounded on $I\backslash\left\{\mu\right\} $:

Since $\left|f\left(x\right)\right|=O\left(\left|x\right|^{n}\right)$
and $\left|x-\mu\right|^{\alpha}+\left|x-\mu\right|^{n}=\Theta\left(\left|x\right|^{n}\right)$
at $x\to\infty$, there exists $d_{1}$ that $g\left(x\right)$ is
bounded on $\left|x-\mu\right|\geq d_{1}$. Also, at $x\to\mu$, since
$\left|f\left(x\right)-f\left(\mu\right)\right|=O\left(\left|x-\mu\right|^{\alpha}\right)$
and $\left|x-\mu\right|^{\alpha}+\left|x-\mu\right|^{n}=\Theta\left(\left|x-\mu\right|^{\alpha}\right)$,
there exists $d_{2}<d_{1}$ such that $g\left(x\right)$ is bounded
on $\left|x-\mu\right|\leq d_{2}$. Since the set $d_{1}\leq\left|x-\mu\right|\leq d_{2}$ is compact, the numerator is bounded on this set, and the denominator is bounded from below by $d_{2}^{\alpha}+d_{2}^{n}$, $g\left(x\right)$
is therefore bounded on $d_{1}\leq\left|x-\mu\right|\leq d_{2}$. In summary,
$g\left(x\right)$ is bounded on $\mathbb{R}\backslash\left\{0\right\}$.

Let $M=\sup_{x\in I\backslash\{\mu\}}\frac{\left|f\left(x\right)-f\left(\mu\right)\right|}{\left|x-\mu\right|^{\alpha}+\left|x-\mu\right|^{n}}$,
we then have:
\[
\left|f\left(x\right)-f\left(\mu\right)\right|=\left(\left|x-\mu\right|^{\alpha}+\left|x-\mu\right|^{n}\right)\cdot g\left(x\right)\leq M\left(\left|x-\mu\right|^{\alpha}+\left|x-\mu\right|^{n}\right)
\]
So the Jensen gap is
\begin{multline*}
\left|\E{f\left(X\right)}-f\left(\E{X}\right)\right|=\int_{\mathbb{R}}\left|f\left(X\right)-f\left(\mu\right)\right|d\mathcal{P}\left(X\right)\\
\leq M\int_{\mathbb{R}}\left|X-\mu\right|^{\alpha}+\left|X-\mu\right|^{n}d\mathcal{P}\left(X\right)\leq M\left(\sigma_{\alpha}^{\alpha}+\sigma_{n}^{n}\right)
\end{multline*}
Also note that $\sigma_{\alpha}\leq\sigma_{n}$ for $\alpha\leq n$,
we then have
\[
M\left(\sigma_{\alpha}^{\alpha}+\sigma_{n}^{n}\right)\leq M\left(1+\sigma_{n}^{n-\alpha}\right)\sigma_{n}^{\alpha}
\]
\end{proof}

If we are only interested in distributions concentrated around $\mu$, we can further simplify the inequality to the corollary below:

\begin{cor}
For functions that satisfy the condition in Theorem \ref{thm:main-result-upper}, there exists a positive number $M'$ independent of the
distribution such that
\begin{equation}
\left|\E{f\left(X\right)}-f\left(\mu\right)\right|\leq M'\sigma_{n}^{\alpha}\label{eq:sigma-alpha-bound}
\end{equation}
for sufficiently small $\sigma_{n}$, 
\end{cor}

\subsection{Tightness of  upper bound\label{subsec:upper-tightness}}

 We show that modulo the preceding constant $M'$, the inequality \ref{eq:sigma-alpha-bound} is sharp.

\begin{prop}\label{prop:upper-sharp}
Let $f\left(x\right)$ be a function that satisfies the condition in Theorem \ref{thm:main-result-upper} with $I=\mathbb{R}$ and has $\left|f\left(x\right)-f\left(\mu\right)\right|\geq M\left|x-\mu\right|^{\alpha}$
 on $x\in\mathbb{R}$ for some $M>0$. Then for any $\sigma_{n}>0$
there exists probability distribution $\mathcal{P}$ that makes 
\[
\left|\E{f\left(X\right)}-f\left(\E{X}\right)\right|\geq M\sigma_{n}^{\alpha}
\]
\end{prop}

\begin{proof}
Let $\distr$ be discrete with
\[
\mathcal{P}\left(\left\{ \mu+\sigma_n\right\} \right)=\mathcal{P}\left(\left\{ \mu-\sigma_n\right\} \right)=\frac{1}{2}
\]
\[
\mathcal{P}\left(\mathbb{R}\backslash\left\{ \mu+\sigma_n,\mu-\sigma_n\right\} \right)=0
\]
The Jensen gap can then be written as
\[
\left|\E{f\left(X\right)}-f\left(\E{X}\right)\right|\geq M\int\left|X-\mu\right|^{\alpha}d\mathcal{P}\left(X\right)=M\sigma_{n}^{\alpha}
\]
\end{proof}

The following proposition shows that the $\sigma_{n}$ in inequality \eqref{eq:sigma-alpha-bound} cannot be replaced by $\sigma_{\beta}$ for any $\beta<n$:

\begin{prop}
There exists a function $f$ that satisfies the condition in Theorem \ref{thm:main-result-upper}
such that for any $0<\beta<n$ and $\sigma_{n}>0$, there exists a probability distribution $\distr$ that makes
$\frac{\left|\jg\right|}{\sigma_{\beta}{}^{\alpha}}$
arbitrarily large.
\end{prop}

\begin{proof}
Let $\mathcal{P}$ be  discrete with
\[
\mathcal{P}\left(\left\{ \mu\right\} \right)=1-p
\]
\[
\mathcal{P}\left(\left\{ \mu+a\right\} \right)=\mathcal{P}\left(\left\{ \mu-a\right\} \right)=p/2
\]
\[
\mathcal{P}\left(\mathbb{R}\backslash\left\{ \mu,\mu+a,\mu-a\right\} \right)=0
\]
Then $\sigma_{\beta}$ can be written as
\[
\sigma_{\beta}=\sqrt[\beta]{p}\cdot a
\]
Let $f\left(x\right)=\left|x-\mu\right|^{\alpha}+\left|x-\mu\right|^{n}$.
The absolute value of the Jensen gap can be written as 
\[
\left|\jg\right|=p\cdot\left(a^{\alpha}+a^{n}\right).
\]
Then the ratio
\[
\frac{\left|\jg\right|}{\sigma_\beta^{\alpha}}=p^{1-\frac{\alpha}{\beta}}\cdot\left(1+a^{n-\alpha}\right).
\]
Note that $\sigma_{n}=\sqrt[n]{p}\cdot a$. We then have $a=\frac{\sigma_{n}}{\sqrt[n]{p}}$.
Then we can write the ratio as
\[
\frac{\left|\jg\right|}{\sigma_\beta^{\alpha}}=p^{1-\frac{\alpha}{\beta}}\cdot\left(1+\frac{\sigma_{n}^{n-\alpha}}{p^{1-\frac{\alpha}{n}}}\right)=p^{1-\frac{\alpha}{\beta}}+p^{\alpha\left(\frac{1}{n}-\frac{1}{\beta}\right)}\cdot\sigma_{n}^{n-\alpha}.
\]
Since $\frac{1}{n}-\frac{1}{\beta}<0$ and $p$ can take any value
in $\left(0,1\right)$, it is always possible to make the ratio arbitrarily
large.
\end{proof}

\subsection{Expanding the scope of the upper bound by linear shifts  \label{subsec:upper-shifting}}

When referring to random variables with distribution peaked around its mean, i.e. random variables with small $\sigma_{n}$, the larger the $\alpha$ in inequality \eqref{eq:sigma-alpha-bound}, the tighter the upper bound. However, for many $f$, it is impossible to find an $\alpha>1$.
For example, for functions that are differentiable at $\mu$
and have a $f'\left(\mu\right)\neq0$, the largest $\alpha$ we can
 obtain is $\alpha=1$. Also, for the case of convex functions that are strictly increasing at $x=\mu$, it is impossible to find an $\alpha>1$:

\begin{prop}
Let $f\left(x\right)$ be a convex function that is strictly increasing
near $\mu$. Then for any $\alpha>1$, we have 
\[
\lim_{x\to\mu}\frac{\left|x-\mu\right|^{\alpha}}{f\left(x\right)-f\left(\mu\right)}=0
\]
\label{prop:convex-alpha-1}
\end{prop}

\begin{proof}
Since $f$ is convex and strictly increasing, we have $f'_+(\mu)>0$ and $f'_-(\mu)>0$. So 
\[
\lim_{x\to\mu^+}\frac{\left|x-\mu\right|^{\alpha}}{f\left(x\right)-f\left(\mu\right)}=\lim_{x\to\mu^+}\left(x-\mu\right)^{\alpha-1}\cdot\frac{x-\mu}{f\left(x\right)-f\left(\mu\right)}=0\cdot\frac{1}{f'_+(\mu)}=0
\]
Same argument holds for $x\to\mu^-$.
\end{proof}

Although the inability to get an $\alpha>1$ seems to be a major limitation, fortunately for most cases we can eliminate this limitation by shifting the function by a linear function, because  this does not change its convexity or the Jensen gap. For functions that are differentiable at $x=\mu$, from Taylor's theorem with Peano's form of remainder, we know that
\[
f\left(x\right)=f\left(\mu\right)+f'\left(\mu\right)\left(x-\mu\right)+o\left(x-\mu\right)
\]
We can therefore study $g\left(x\right)=f\left(x\right)-f'\left(\mu\right)\left(x-\mu\right)$
instead of $f\left(x\right)$. We will then have $g\left(x\right)-g\left(\mu\right)=o\left(x-\mu\right)$,
which has an $\alpha$ value at least as large as $f\left(x\right)$.
If further $f\left(x\right)$ has well defined second derivative, we then have
\[
f\left(x\right)=f\left(\mu\right)+f'\left(\mu\right)\left(x-\mu\right)+\frac{f^{''}\left(\xi_{L}\right)}{2}\left(x-\mu\right)^{2}
\]
that is 
\begin{equation}
g\left(x\right)-g\left(\mu\right)=\frac{f^{''}\left(\xi_{L}\right)}{2}\left(x-\mu\right)^{2}\label{eq:gx}
\end{equation}
which implies $\alpha=2$. If $f''(\mu)=0$, we can apply similar arguments to higher order derivatives to find the best $\alpha$.

Note that if we define $h\left(x;\mu\right)\equiv\frac{f''\left(\xi_{L}\right)}{2}\equiv\frac{f\left(x\right)-f\left(\mu\right)-f'\left(\mu\right)\left(x-\mu\right)}{\left(x-\mu\right)^{2}}$,
then \eqref{eq:gx} can be written as $g\left(x\right)-g\left(\mu\right)=h\left(x;\mu\right)\left(x-\mu\right)^{2}$,
which further gives
\begin{equation}
\inf_{x}h\left(x;\mu\right)\cdot\mathrm{Var}\left[X\right]\leq\J{f}\leq\sup_{x}h\left(x;\mu\right)\cdot\mathrm{Var}\left[X\right]\label{eq:liaoresult}
\end{equation}
as shown in\cite{liao2017sharpening}. If $\left|f\left(x\right)\right|\neq O\left(x^{2}\right)$
at $x\to\infty$, then $\sup_{x}h\left(x;\mu\right)=+\infty$ or $\inf_{x}h\left(x;\mu\right)=-\infty$
or both, which means at least half of \eqref{eq:liaoresult} will become
a trivial inequality $-\infty\leq\J{f}$ or $\J{f}\leq+\infty$.
On the other hand, if $\left|f\left(x\right)\right|=O\left(x^{2}\right)$
at $x\to\infty$, we then have $n=2$. If this is the case, the preceding
constant $M$ in Theorem \ref{thm:main-result-upper} can then be written as $M=\frac{1}{2}\sup_{x}\left|h\left(x;\mu\right)\right|$
and Equation (2.1) therefore becomes $-\sup_{x}\left|h\left(x;\mu\right)\right|\cdot\sigma_{2}^{2}\leq\mathcal{J}\leq\sup_{x}\left|h\left(x;\mu\right)\right|\cdot\sigma_{2}^{2}$,
which is equivalent to \eqref{eq:liaoresult} in half or in full$^2$\footnotemark[2]\footnotetext[2]{$^2$If $\left|\sup_{x}h\left(x;\mu\right)\right|>\left|\inf_{x}h\left(x;\mu\right)\right|$,
then we must have $\sup_{x}\left|h\left(x;\mu\right)\right|=\sup_{x}h\left(x;\mu\right)$,
which means the $\J{f}\leq$ part of Theorem \ref{thm:main-result-upper} and of \eqref{eq:liaoresult}
are equivalent. If $\left|\sup_{x}h\left(x;\mu\right)\right|<\left|\inf_{x}h\left(x;\mu\right)\right|$,
then we must have $-\sup_{x}\left|h\left(x;\mu\right)\right|=\inf_{x}h\left(x;\mu\right)$,
which means the $\leq\J{f}$ part of Theorem \ref{thm:main-result-upper} and of \eqref{eq:liaoresult}
are equivalent. If $\left|\sup_{x}h\left(x;\mu\right)\right|=\left|\inf_{x}h\left(x;\mu\right)\right|$
and $h\left(x;\mu\right)$ is not constant, then we must have $\sup_{x}\left|h\left(x;\mu\right)\right|=-\inf_{x}h\left(x;\mu\right)=\sup_{x}h\left(x;\mu\right)$,
hence in this case, Theorem \ref{thm:main-result-upper} is fully equivalent to \eqref{eq:liaoresult}.}.
Due to these connections, Lemma 1 in \cite{liao2017sharpening}
gives a convenient way to compute the $M$ in equation \eqref{eq:sigma-alpha-n-bound} when the $f'\left(x\right)$ is convex or concave.

\section{Second Main Result: Lower bound}
\label{subsec:lower-main} 
 We first prove our lower bound  for conditions similar to the upper bound case.  The tightness of this bound will be discussed in Section \ref{subsec:lower-tight} followed, in Section \ref{subsec:lower-convex}, by strong implications for convex functions, and expanding the scope via linear function shifts.

 The lower bound  given in the following theorem holds for any probability distribution as long as the relevant moments are well-defined.

\begin{theorem}\label{thm:main-result-lower}
If function $f:I\to\mathbb{R}$, where $I$ is a closed subset of $\mathbb{R}$ and $\mu\in I$, satisfies the following conditions:
\begin{enumerate}
\item $f(x)-f(\mu)>0$ at $x\neq\mu$
\item $f\left(x\right)-f\left(\mu\right)=\Omega\left(\left|x-\mu\right|^{\alpha}\right)$
at $x\to\mu$ for $\alpha>0$
\item $f\left(x\right)-f\left(\mu\right)=\Omega\left(\left|x-\mu\right|^{\beta}\right)$
at $x\to\infty$ for $0\leq\beta\leq\alpha$
\end{enumerate}
then  for random variable $X$ with probability distribution $\distr$ that has expectation $\mu$, the following inequality holds:
\begin{equation}
\jg\geq M\frac{\sigma_{\alpha/2}^{\alpha}}{1+\sigma_{\alpha-\beta}^{\alpha-\beta}}\label{eq:cauchy-schwartz-lower-bound}
\end{equation}
where $M=\inf_{x\in I\backslash \{\mu\}}\left\{ \left[f\left(x\right)-f\left(\mu\right)\right]\cdot\left(\frac{1}{\left|X-\mu\right|^{\beta}}+\frac{1}{\left|X-\mu\right|^{\alpha}}\right)\right\}>0$ does not depend on the probability distribution $\distr$.
\end{theorem}

\begin{proof}
Let 
\[
g\left(x\right)=\left(\frac{1}{\left|x-\mu\right|^{\beta}}+\frac{1}{\left|x-\mu\right|^{\alpha}}\right)^{-1}
\]
from the definition of $M$, we know that $f\left(x\right)-f\left(\mu\right)\geq M\cdot g(x)$.

 We first prove that $M>0$.  It is easy to see that $g\left(x\right)$ is positive at $x\neq\mu$, $g\left(x\right)=\Theta\left(\left|x-\mu\right|^{\alpha}\right)$ at $x\to\mu$,
and $g\left(x\right)=\Theta\left(\left|x-\mu\right|^{\beta}\right)$ at $x\to\infty$.
Therefore, there exists positive  $M_{1}$, $M_{2}$ and $d_{1}\leq d_{2}$
such that $f\left(x\right)-f\left(\mu\right)\geq M_{1}\cdot g\left(x\right)$
at $\left|x-\mu\right|\leq d_{1}$ and $f\left(x\right)-f\left(\mu\right)\geq M_{2}\cdot g\left(x\right)$
at $\left|x-\mu\right|\geq d_{2}$. Since $d_{1}\leq\left|x-\mu\right|\leq d_{2}$ is compact and both $f\left(x\right)-f\left(\mu\right)$ and $g\left(x\right)$
are positive in this interval, there exists $M_{3}>0$ such that $f\left(x\right)-f\left(\mu\right)\geq M_{3}\cdot g\left(x\right)$.
Taking $M'=\min\left\{ M_{1},M_{2},M_{3}\right\}>0$, we have $f\left(x\right)-f\left(\mu\right)\geq M'g\left(x\right)$.
That is, $\frac{f\left(x\right)-f\left(\mu\right)}{g(x)}$ is bounded from below by some positive number. Therefore
\[
M=\inf_{x\in I\backslash \{\mu\}}\frac{f\left(x\right)-f\left(\mu\right)}{g(x)}>0
\]

Since $f\left(x\right)-f\left(\mu\right)\geq M\cdot g\left(x\right)$, we have
\begin{multline}
\jg\geq M\int\left(\frac{1}{\left|X-\mu\right|^{\beta}}+\frac{1}{\left|X-\mu\right|^{\alpha}}\right)^{-1}d\mathcal{P}\left(X\right)\\
=M\int\frac{\left|X-\mu\right|^{\alpha}}{\left|X-\mu\right|^{\alpha-\beta}+1}d\mathcal{P}\left(X\right)\label{eq:ineq-int}
\end{multline}
The Cauchy\textendash Schwarz
inequality can be used to simplify the above inequality: Let 
\[
g_{1}\left(X\right)=\sqrt{\frac{\left|X-\mu\right|^{\alpha}}{\left|X-\mu\right|^{\alpha-\beta}+1}}
\]
and
\[
g_{2}\left(X\right)=\sqrt{\left|X-\mu\right|^{\alpha-\beta}+1}
\]
Cauchy\textendash Schwarz inequality  
\[
\E{g_{1}^{2}\left(X\right)}\E{g_{2}^{2}\left(X\right)}\geq\E{g_{1}\left(X\right)g_{2}\left(X\right)}^{2}
\]
  can be rewritten as
\[
\E{g_{1}^{2}\left(X\right)}\geq\frac{\E{g_{1}\left(X\right)g_{2}\left(X\right)}^{2}}{\E{g_{2}^{2}\left(X\right)}}
\]
Note that
\[
\E{g_{1}^{2}\left(X\right)}=\int\frac{\left|X-\mu\right|^{\alpha}}{\left|X-\mu\right|^{\alpha-\beta}+1}d\mathcal{P}\left(X\right)
\]
\[
\E{g_{2}^{2}\left(X\right)}=\int\left(\left|X-\mu\right|^{\alpha-\beta}+1\right)d\mathcal{P}\left(X\right)=1+\sigma_{\alpha-\beta}^{\alpha-\beta}
\]
\[
\E{g_{1}\left(X\right)g_{2}\left(X\right)}^{2}=\left(\int\left|X-\mu\right|^{\alpha/2}d\mathcal{P}\left(X\right)\right)^{2}=\sigma_{\alpha/2}^{\alpha}
\]
We therefore have
\[
\int\frac{\left|X-\mu\right|^{\alpha}}{\left|X-\mu\right|^{\alpha-\beta}+1}d\mathcal{P}\left(X\right)\geq\frac{\sigma_{\alpha/2}^{\alpha}}{1+\sigma_{\alpha-\beta}^{\alpha-\beta}}
\]
Plugging into \eqref{eq:ineq-int}, we have
\begin{equation*}
\jg\geq M\frac{\sigma_{\alpha/2}^{\alpha}}{1+\sigma_{\alpha-\beta}^{\alpha-\beta}}
\end{equation*}
\end{proof}
Note that if we replace all $f(x)-f(\mu)$ with $f(\mu)-f(x)$, and at the same time replace $\jg$ with $-\jg$, Theorem \ref{thm:main-result-lower} still holds. Also note that by replacing the Cauchy\textendash Schwarz inequality with
H\"older's inequality in the proof of the above theorem, we can get
a more general but less pleasing result:
\begin{theorem}
The inequality \eqref{eq:cauchy-schwartz-lower-bound} in Theorem \ref{thm:main-result-lower}
can be replaced by the following inequality:
\begin{equation}\label{eq:lower-bound-holder}
\jg\geq M\frac{\left[\sum_{l=0}^{\left(k+1\right)/q-1}\left(\begin{array}{c}
\left(k+1\right)/q-1\\
l
\end{array}\right)\sigma_{\alpha/p+l\left(\alpha-\beta\right)}^{\alpha/p+l\left(\alpha-\beta\right)}\right]^{p}}{\left[\sum_{l=0}^{k}\left(\begin{array}{c}
k\\
l
\end{array}\right)\sigma_{l\left(\alpha-\beta\right)}^{l\left(\alpha-\beta\right)}\right]^{p/q}}
\end{equation}
where $k\geq1$ is an integer, $q$ can be any positive factor of
$\left(k+1\right)$ except $1$, and $p=\frac{q}{q-1}$.
\end{theorem}
\begin{proof}
Following the same steps as in the Cauchy\textendash Schwarz version,
but this time introducing a new integral parameter $k\geq1$, and setting
\[
g_{1}\left(X\right)=\sqrt[p]{\frac{\left|X-\mu\right|^{\alpha}}{\left|X-\mu\right|^{\alpha-\beta}+1}}
\]
and
\[
g_{2}\left(X\right)=\sqrt[q]{\left(\left|X-\mu\right|^{\alpha-\beta}+1\right)^{k}}
\]
we have
\[
\E{g_{1}^{p}\left(X\right)}=\int\frac{\left|X-\mu\right|^{\alpha}}{\left|X-\mu\right|^{\alpha-\beta}+1}d\mathcal{P}\left(X\right)
\]
\[
\E{g_{2}^{q}\left(X\right)}=\int\left(\left|X-\mu\right|^{\alpha-\beta}+1\right)^{k}d\mathcal{P}\left(X\right)=\sum_{l=0}^{k}\left(\begin{array}{c}
k\\
l
\end{array}\right)\sigma_{l\left(\alpha-\beta\right)}^{l\left(\alpha-\beta\right)}
\]
\begin{multline*}
\E{g_{1}\left(X\right)g_{2}\left(X\right)}=\int\left|X-\mu\right|^{\alpha/p}\left(\left|X-\mu\right|^{\alpha-\beta}+1\right)^{k/q-1/p}d\mathcal{P}\left(X\right)\\
=\int\left|X-\mu\right|^{\alpha/p}\left(\left|X-\mu\right|^{\alpha-\beta}+1\right)^{\left(k+1\right)/q-1}d\mathcal{P}\left(X\right)\\
=\sum_{l=0}^{\left(k+1\right)/q-1}\left(\begin{array}{c}
\left(k+1\right)/q-1\\
l
\end{array}\right)\sigma_{\alpha/p+l\left(\alpha-\beta\right)}^{\alpha/p+l\left(\alpha-\beta\right)}.
\end{multline*}
From H\"older's inequality, we know that
\[
\E{g_{1}^{p}\left(X\right)}\geq\frac{\E{g_{1}\left(X\right)g_{2}\left(X\right)}^{p}}{\E{g_{2}^{q}\left(X\right)}^{\frac{p}{q}}}=\frac{\left[\sum_{l=0}^{\left(k+1\right)/q-1}\left(\begin{array}{c}
\left(k+1\right)/q-1\\
l
\end{array}\right)\sigma_{\alpha/p+l\left(\alpha-\beta\right)}^{\alpha/p+l\left(\alpha-\beta\right)}\right]^{p}}{\left[\sum_{l=0}^{k}\left(\begin{array}{c}
k\\
l
\end{array}\right)\sigma_{l\left(\alpha-\beta\right)}^{l\left(\alpha-\beta\right)}\right]^{p/q}}
\]
which immediately yields inequality \eqref{eq:lower-bound-holder}.
\end{proof}

Although general, inequality \eqref{eq:lower-bound-holder} is too cumbersome to be useful. To simplify it, we can take $q=k+1$ and therefore $p=1+\frac{1}{k}$. We then have
\begin{equation}
\label{eq:lower-bound-holder-special-q}
\jg\geq M\frac{\sigma_{\alpha/\left(1+\frac{1}{k}\right)}^{\alpha}}{\left[\sum_{l=0}^{k}\left(\begin{array}{c}
k\\
l
\end{array}\right)\sigma_{l\left(\alpha-\beta\right)}^{l\left(\alpha-\beta\right)}\right]^{1/k}}
\end{equation}

Note that applying inequality \eqref{eq:lower-bound-holder-special-q} to $f(x)=\left|x\right|^\alpha$, we obtain
\begin{equation}
\label{eq:lower-bound-special}
\sigma_\alpha\geq\sigma_{\alpha/\left(1+\frac{1}{k}\right)}
\end{equation}
which is a special case of the inequality 
\[
\E{\left|X\right|^{r}}\leq\E{\left|X\right|^{s}}^{\frac{r}{s}}
\]
for $0<r<s$.

\subsection{Tightness of the lower bound\label{subsec:lower-tight}}

 Since inequality \eqref{eq:lower-bound-special} is a special case of \eqref{eq:lower-bound-holder-special-q} and since \eqref{eq:lower-bound-special} is sharp, it follows that \eqref{eq:lower-bound-holder-special-q} is sharp. 
 
In inequality \eqref{eq:lower-bound-holder-special-q}, as the centered absolute moments decrease to $0$, since the denominator  decreases to $1$, it is the numerator that characterizes how fast the Jensen gap decreases to $0$.
Since $\sigma_r\leq\sigma_s$ for $r\leq s$, having a larger subscript in the numerator means a tighter result.
In \eqref{eq:lower-bound-holder-special-q}, the subscript of the numerator can be increased to a value arbitrarily close to $\alpha$ by choosing larger $k$, but as a side effect, this also brings higher orders of moments into the denominator.

 Therefore, a natural question is whether we can increase the subscript of the numerator of \eqref{eq:lower-bound-holder-special-q} without bringing in higher orders of moments? The following proposition shows that the answer is no, by showing that if we increase the subscript higher than that proposed in \eqref{eq:lower-bound-holder-special-q}, it is possible to construct a sequence of probability distributions that make the moments in the denominator decrease to 0 while at the same time making the ratio between the Jensen gap and the numerator go to zero (therefore making it impossible to find a $M$ to make the $\geq$ in \eqref{eq:lower-bound-holder-special-q} hold):

\begin{prop}
Let $f(x)=\Theta(\left|x\right|^\beta)$ at $x\to\infty$ be a function that satisfies the condition in Theorem \ref{thm:main-result-lower}. Then for any $q>\alpha/\left(1+\frac{1}{k}\right)$, there exists a sequence of probability  distributions $\mathcal{P}^{(1)}, \mathcal{P}^{(2)}, \ldots$ such that $\sigma_{r}^{(j)}$ is non-increasing with respect to $j$ for all $r\leq k(\alpha-\beta)$ and 
\[
\lim_{j\to+\infty}\frac{\mathcal{J}\left(f,X\sim\mathcal{P}^{(j)}\right)}{\left[\sigma_q^{(j)}\right]^\alpha}=0
\]
\end{prop}

\begin{proof}
Let $m=k(\alpha-\beta)$. Without loss of generality,  assume $\mu=0$, $f(x)$ is even, and $f(0)=0$. 
Let $\mathcal{P}$ be a discrete probability distribution that has
\[
\mathcal{P}\left(\left\{j\right\} \right)=\mathcal{P}\left(\left\{ -j\right\} \right)=\frac{1}{2j^m}
\]
\[
\mathcal{P}\left(\left\{0\right\}\right)=1-\frac{1}{j^m}
\]
\[
\mathcal{P}\left(\mathbb{R}\backslash\left\{ 0,\pm j\right\} \right)=0
\]
It is easy to see that
\[
\sigma_r^{(j)}=j^{1-\frac{m}{r}}
\]
does not increase as $j$ increases for $r\leq m$, and
\[
\mathcal{J}\left(f,X\sim\mathcal{P}^{(j)}\right)=\frac{f(j)}{j^m}=\Theta\left(j^{\beta-m}\right)
\]
at $j\to+\infty$. We then have
\[
\frac{\mathcal{J}\left(f,X\sim\mathcal{P}^{(j)}\right)}{\left[\sigma_q^{(j)}\right]^\alpha}=\Theta\left[j^{\beta-m-\alpha\left(1-\frac{m}{q}\right)}\right]
\]
In the case that $q>\alpha/\left(1+\frac{1}{k}\right)$, we have $\beta-m-\alpha\left(1-\frac{m}{q}\right)<0$. Therefore
\[
\Theta\left[j^{\beta-m-\alpha\left(1-\frac{m}{q}\right)}\right]\to0
\]
as $j\to+\infty$.
\end{proof}

\subsection{The lower bound for convex functions\label{subsec:lower-convex}}

The conditions for Theorem \ref{thm:main-result-lower} are hard for a general function to satisfy. In fact, Jensen's inequality only holds for convex functions, so it is not surprising that we are unable to obtain a lower bound of the Jensen gap. In this section, we show that convexity implies the conditions in Theorem \ref{thm:main-result-lower}. The argument in this section also applies to concave functions. 

In order for the condition in Theorem \ref{thm:main-result-lower}  to be satisfied, a convex function $f(x)$ needs to be non-increasing at $\left(-\infty,\mu\right]$ and non-decreasing at $\left[\mu,+\infty\right)$. The following proposition shows that  it is always possible to shift a convex function by a linear function to make it so.

\begin{prop}
For any convex function $f\left(x\right)$, and any real number $a$ satisfying $f'_-(\mu)\leq a\leq f'_+(\mu)$, the linear shift $g\left(x\right)=f\left(x\right)-a\left(x-\mu\right)$ is
non-increasing at $\left(-\infty,\mu\right]$ and non-decreasing at
$\left[\mu,+\infty\right)$. Specially, if $f\left(x\right)$ is differentiable
at $\mu$, $a$ is unique and given by $a=f'\left(x\right)$.
\end{prop}

\begin{proof}
For $\mu \leq x < x'$, we have
\[
\frac{g(x')-g(x)}{x'-x} = \frac{f(x')-f(x)}{x'-x} - a \geq f'_+(\mu)-a\geq0
\]
That is, $g(x')\geq g(x)$. Similar argument applies to the $x\leq\mu$ half of $g(x)$.
\end{proof}

The convexity also implies that $\alpha$ can only  take values from $\left[1,+\infty\right)$, as shown in the following proposition:

\begin{prop}
There does not exist any convex function that has $f\left(x\right)-f\left(\mu\right)=\Omega\left(\left|x-\mu\right|^{\alpha}\right)$
at $x\to\mu$ with $\alpha<1$.
\end{prop}
\begin{proof}
Since $f\left(x\right)-f\left(\mu\right)=\Omega\left(\left|x-\mu\right|^{\alpha}\right)$
as $x\to\mu$, there exists positive number $d$ and $M$ such that
$f\left(x\right)-f\left(\mu\right)\geq M\left|x-\mu\right|^{\alpha}$
at $\left|x-\mu\right|\leq d$.
Then for any $x$ that has $\mu<x<\mu+d$, we have
\[
\frac{f\left(x\right)-f\left(\mu\right)}{x-\mu}\geq M\left(x-\mu\right)^{\alpha-1}
\]
Since $\frac{f\left(x\right)-f\left(\mu\right)}{x-\mu}$ is non-decreasing with respect to $x$, if $\alpha<1$, $\left(x-\mu\right)^{\alpha-1}$ will becomes arbitrarily
high as $x\to\mu^+$, making it impossible to for the above inequality to be true as $x$ decreases to $\mu$.
\end{proof}

The following proposition shows that for convex functions, it is possible to find a $\beta$ at least as large as $1$:

\begin{prop}
If $f\left(x\right)$ is convex, then at $x\to\infty$, $\left|f\left(x\right)\right|$
is either constant or $\Omega\left(\left|x\right|\right)$.
\end{prop}
\begin{proof}
If $f\left(x\right)$ is constant, then this proposition automatically
holds true. Otherwise, let $x_{0}<x_{1}$ be two real numbers such
that $f\left(x_{0}\right)\neq f\left(x_{1}\right)$. Without loss
of generality, let us assume $f\left(x_{0}\right)<f\left(x_{1}\right)$.
Since $f\left(x\right)$ is convex, then for` any $x>x_{1}$
\[
\frac{f\left(x\right)-f\left(x_{0}\right)}{x-x_{0}}\geq\frac{f\left(x_{1}\right)-f\left(x_{0}\right)}{x_{1}-x_{0}}>0
\]
That is
\[
f\left(x\right)-f\left(x_{0}\right)\geq\frac{f\left(x_{1}\right)-f\left(x_{0}\right)}{x_{1}-x_{0}}\cdot\left(x-x_{0}\right)
\]
i.e.,  $f\left(x\right)=\Omega\left(x\right)$ at $x\to+\infty$.
Considering all the cases, i.e. when $x\to-\infty$, and when $f\left(x_{0}\right)>f\left(x_{1}\right)$, we get $\left|f\left(x\right)\right|=\Omega\left(\left|x\right|\right)$.
\end{proof}

Although for the lower bound case our result has no similar equivalence
relation with \cite{liao2017sharpening} as the one discussed in Section \ref{subsec:upper-shifting}, the preceding constant $M$ in our Theorem \ref{thm:main-result-lower}, after the linear
shift, can be written as $M=2\cdot\sup_{x\neq\mu}h\left(x;\mu\right)$
when $\alpha=\beta=2$. For this special case, when $f'\left(x\right)$
is convex or concave, the Lemma 1 in \cite{liao2017sharpening} is
still helpful.

\section{Further Discussion, Conclusion and Open problems}

The procedure in the proofs of Theorems \ref{thm:main-result-upper} and \ref{thm:main-result-lower} can be thought
of as a general scheme, which is also followed by \cite{liao2017sharpening},
of obtaining bounds on the Jensen gap. This procedure first writes
$f\left(x\right)-f\left(\mu\right)$ as a product of two functions,
say $s\left(x\right)t\left(x\right)$, where the $\sup$ and $\inf$
of $s$ are easy to compute and the integral $\int t\left(X\right)d\mathcal{P}\left(X\right)$
can be easily computed or further bounded. We then have
\[
\inf s\left(x\right)\cdot\int t\left(X\right)d\mathcal{P}\left(X\right)\leq\J{f}\leq\sup s\left(x\right)\cdot\int t\left(X\right)d\mathcal{P}\left(X\right)
\]
The above formula gives a very general way to bound the Jensen gap.
For example, instead of using $t\left(x\right)=\left|x-\mu\right|^{\alpha}+\left|x-\mu\right|^{n}$
as in Theorem \ref{thm:main-result-upper}, the reader can choose a more general form $t\left(x\right)=\sum_{\alpha\leq\eta\leq n}a_{\eta}\left|x-\mu\right|^{\eta}$
where values of $\eta$ and $a_{\eta}$ are chosen based on the application
to better approximate $f\left(x\right)$, and obtain an improved upper
bound
\[
\mathcal{J}\leq\sup\frac{f\left(x\right)}{t\left(x\right)}\cdot\left(\sum_{\alpha\leq\eta\leq n}a_{\eta}\sigma_{\eta}^{\eta}\right)
\]
Similarly, instead of using $t\left(x\right)=\left(\frac{1}{\left|x-\mu\right|^{\alpha}}+\frac{1}{\left|x-\mu\right|^{\beta}}\right)^{-1}$
as in Theorem \ref{thm:main-result-lower}, the reader can choose $t\left(x\right)=\left(\sum_{\beta\leq\eta\leq\alpha}\frac{a_{\eta}}{\left|x-\mu\right|^{\eta}}\right)^{-1}$
where values of $\eta$ and $a_{\eta}$ depend on the application,
and obtain an improved lower bound
\[
\mathcal{J}\geq\inf\frac{f\left(x\right)}{t\left(x\right)}\cdot\frac{\sigma_{\alpha/2}^{\alpha}}{\sum_{\beta\leq\eta\leq\alpha}a_{\eta}\sigma_{\alpha-\eta}^{\alpha-\eta}}
\]
or
\[
\mathcal{J}\geq\inf\frac{f\left(x\right)}{t\left(x\right)}\cdot\frac{\sigma_{\alpha/\left(1+\frac{1}{k}\right)}^{\alpha}}{\left(\sum_{\beta\leq\eta_{1},\cdots,\eta_{k}\leq\alpha}a_{\eta_{1}}\cdots a_{\eta_{k}}\sigma_{k\alpha-\eta_{1}-\cdots-\eta_{k}}^{k\alpha-\eta_{1}-\cdots-\eta_{k}}\right)^{1/k}}
\]

We have obtained general upper and lower bounds on Jensen's gap that depend on the asymptotic growth of the function and related moments of the random variable's distribution and compared the new bounds with existing upper and lower bounds.
Although fairly general, some conditions in our theorems are still too strong for some situations. For example, in our upper bound, we require the function to grow no faster than polynomial at $x\to\infty$, which excludes some useful functions, such as exponential functions. Also, we require the function to be bounded on any compact subset of $\mathbb{R}$ in our upper bound, which exclude the study of the Jensen gap for functions like $\log(x)$, $\frac{1}{x}$  with random variable $X$ on $\left(0,+\infty\right)$.  Future work is proposed to extend our results to include such cases. 

\section{Declarations}

\subsection{Availability of data and material}

Not applicable

\subsection{Competing interests}

The authors declares no competing interests.

\subsection{Funding}

This work is partly funded by National Institutes of Health \newline [grant number R01GM110077].

\subsection{Authors' contributions}

XG proves most theorems and is responsible for most part of paper writing.
MS has major contribution on the organization of the paper and paper writing of the whole paper, and provides critical insights on Section \ref{subsec:lower-convex} and Proposition \ref{prop:convex-alpha-1}.
AER has contribution on technical discussion and supervision of the overall project.

\subsection{Acknowledgement}

The authors would like to thank Justin S. Smith for a thorough proof reading to correct grammatical errors.
The authors would like to thank J.G. Liao for the discussion on the
connection between our results and his work \cite{liao2017sharpening}.

\subsection{Authors information}

Author's email:

Xiang Gao: qasdfgtyuiop@gmail.com

Meera Sitharam: sitharam@cise.ufl.edu

Adrian E. Roitberg: roitberg@ufl.edu

\end{document}